\documentclass{amsart}

\usepackage{amssymb,enumitem}
\usepackage[bookmarks]{hyperref}
\usepackage[all,cmtip]{xy}
\usepackage{cite}

\newcommand{\F}{\mathbb{F}}
\newcommand{\G}{\mathbb{G}}
\newcommand{\N}{\mathbb{N}}
\newcommand{\Z}{\mathbb{Z}}

\DeclareMathOperator{\Ext}{Ext}
\DeclareMathOperator{\opH}{H}
\DeclareMathOperator{\Hbul}{\opH^\bullet}

\DeclareMathOperator{\Hom}{Hom}

\DeclareMathOperator{\ind}{ind}
\DeclareMathOperator{\Lie}{Lie}

\newcommand{\g}{\mathfrak{g}}
\newcommand{\fsl}{\mathfrak{sl}}
\newcommand{\gl}{\mathfrak{gl}}

\newcommand{\gln}{\gl_n}

\newcommand{\glnr}{\gln^{(r)}}
\newcommand{\gltr}{\gl_2^{(r)}}

\newtheorem{theorem}{Theorem}[section]
\newtheorem{proposition}[theorem]{Proposition}

\newtheorem{corollary}[theorem]{Corollary}

\newtheorem{lemma}[theorem]{Lemma}

\newtheorem*{theorem*}{Theorem}
\newtheorem*{lemma*}{Lemma}

\theoremstyle{definition}

\newtheorem{remark}[theorem]{Remark}

\numberwithin{equation}{section}

\title{Universal extension classes for \texorpdfstring{$GL_2$}{GL2}}

\author{Christopher M.\ Drupieski}
\address{Department of Mathematical Sciences \\ DePaul University \\ Chicago, IL 60614, USA}
\email{cdrupies@depaul.edu}


\date{\today}

\begin{document}

\begin{abstract}
In this note we give a new existence proof for the universal extension classes for $GL_2$ previously constructed by Friedlander and Suslin via the theory of strict polynomial functors. The key tool in our approach is a calculation of Parker showing that, for suitable choices of coefficient modules, the Lyndon--Hochschild--Serre spectral sequence for $SL_2$ relative to its first Frobenius kernel stabilizes at the $E_2$-page. Consequently, we obtain a new proof that if $G$ is an infinitesimal subgroup scheme of $GL_2$, then the cohomology ring $\Hbul(G,k)$ of $G$ is a finitely-generated noetherian $k$-algebra.
\end{abstract}

\maketitle

\section{Introduction}

In this note we give a new existence proof for the universal extension classes for $GL_2$ previously constructed by Friedlander and Suslin via the theory of strict polynomial functors. These classes have also been exhibited without recourse to strict polynomial functors by van der Kallen \cite[Lemma 4.7]{Kallen:2004}. The key tool in our approach is a calculation of Parker showing that, for suitable choices of coefficient modules, the Lyndon--Hochschild--Serre spectral sequence for $SL_2$ relative to its first Frobenius kernel stabilizes at the $E_2$-page. Consequently, we obtain a new proof that if $G$ is an infinitesimal subgroup scheme of $GL_2$, then the cohomology ring $\Hbul(G,k)$ of $G$ is a finitely-generated noetherian $k$-algebra. Using Parker's recursive formulas for computing higher extensions between modules for $SL_2$, we also obtain the fact that the universal extension classes for $GL_2$ are unique up to scalar multiples (Theorem \ref{theorem:extclassesGL2}); this is a sharper result than the existence statement given by van der Kallen. Most of the notation in this article is standard, and can be found for example in \cite{Jantzen:2003,Parker:2007}.

\section{Preliminaries} \label{section:preliminaries}

Let $G$ be a reduced algebraic group scheme over $k$; we will be primarily interested in the cases $G = GL_2$ and $G = SL_2$. Assume that $G$ is defined over the prime field $\F_p$, and let $F: G \rightarrow G$ be the standard Frobenius morphism defining the $\F_p$-structure on $G$. For each integer $r \geq 1$, write $G_r$ for the $r$-th infinitesimal Frobenius kernel of $G$, that is, the scheme-theoretic kernel of the $r$-th iterate $F^r$ of $F$. Given a rational $G$-module $V$, let $V^* = \Hom_k(V,k)$ denote the dual module, and let $V^{(r)}$ denote the $r$-th Frobenius twist of $V$, that is, the rational $G$-module obtained by twisting the structure map for $V$ by $F^r$. Then the action of $G_r$ on $V^{(r)}$ is trivial. If $V$ and $W$ are finite-dimensional rational $G$-modules, then there exists a natural isomorphism of rational cohomology groups $\Ext_G^\bullet(V,W) \cong \Ext_G^\bullet(W^*,V^*)$.

Next let $G = SL_2$. Let $T \subset SL_2$ be a maximal torus. Then the character group $X(T)$ is generated as a free abelian group by the fundamental dominant weight $\varpi$. Henceforth, we identify $X(T)$ with $\Z$ via $n \varpi \mapsto n$. Under this identification, the set $X(T)_+$ of dominant weights identifies with the subset $\N$ of non-negative integers. Given $n \in \N$, write $L(n)$, $\Delta(n)$, and $\nabla(n)$, respectively, for the irreducible, Weyl, and induced modules for $SL_2$ of highest weight $n$. Then $L(n)$ occurs as the socle of $\nabla(n)$ and as the head of $\Delta(n)$, and $\Delta(n) \cong \nabla(n)^*$. The modules $\Delta(n)$ and $\nabla(n)$ are also known, respectively, as the standard and costandard modules for $SL_2$ of highest weight $n$. Taking $n = 0$, one has the trivial module $k = L(0) = \Delta(0) = \nabla(0)$.

Write $k^2$ for the natural representation of $GL_2$. As an $SL_2$-module, $k^2 \cong (k^2)^*$, and $k^2 \cong L(1) \cong \Delta(1) \cong \nabla(1)$. The tensor product $k^2 \otimes (k^2)^* \cong \Hom_k(k^2,k^2)$ identifies as a $GL_2$-module with $\gl_2$, the adjoint representation of $GL_2$. Evidently, $\gl_2$ is self-dual as an $SL_2$-module. If $p \neq 2$, then $L(2) \cong \Delta(2) \cong \nabla(2) \cong \fsl_2$, the adjoint representation for $SL_2$, and the tensor product $k^2 \otimes (k^2)^* \cong L(1) \otimes L(1)$ is isomorphic as an $SL_2$-module to the direct sum $L(0) \oplus L(2)$. Indeed, if $p \neq 2$, then the inclusion of the trivial module into $\Hom_k(k^2,k^2)$ as the scalar multiples of the identity splits via the trace map. On the other hand, if $p = 2$, then one still has $\Delta(2) \cong \fsl_2$, but $L(1) \otimes L(1)$ is isomorphic to $T(2)$, the indecomposable tilting module for $SL_2$ of highest weight $2$. At any rate, for all primes there exist short exact sequences of $SL_2$-modules
\begin{align}
0 &\rightarrow \nabla(0) \rightarrow L(1) \otimes L(1) \rightarrow \nabla(2) \rightarrow 0, \quad \text{and} \label{eq:goodfiltration} \\
0 &\rightarrow \Delta(2) \rightarrow L(1) \otimes L(1) \rightarrow \Delta(0) \rightarrow 0; \label{eq:Weylfiltration}
\end{align}
cf.\ \cite[\S1]{Doty:2005}. In particular, $\gl_2 \cong L(1) \otimes L(1)$ admits a Weyl filtration with sections $\Delta(2)$ and $\Delta(0)$.

As is customary, we identify the blocks of $SL_2$ with subsets of $\N$ via $L(n) \mapsto n$. Let $\lambda,\mu \in \N$, and write $\lambda = pa+i$ and $\mu = pb+j$, with $a,b \in \N$ and $0 \leq i,j \leq p-1$. If $i = p-1$, then $\mu$ is in the same block as $\lambda$ only if $j = p-1$. If $i \neq p-1$, then $\mu$ is in the same block as $\lambda$ only if either $a - b$ is even and $i = j$, or if $a - b$ is odd and $j = p-2-i$. In particular, if $s \geq 1$, then $\lambda$ lies in the same block as $2p^s$ only if either $a$ is even and $i = 0$, or if $a$ is odd and $i = p-2$. In this paper we apply various recursive formulas developed by Parker \cite{Parker:2007} for computing the higher extension groups between certain classes of rational $SL_2$-modules. In order to simplify the statements of some formulas, we sometimes include more summands than are written in the formulas' original statements. This causes no harm, since the extra summands will correspond to extension groups between modules whose highest weights lie in different blocks for $SL_2$, and hence will be zero.

\section{Existence of universal extension classes for \texorpdfstring{$SL_2$}{SL2}} \label{section:existence}

The first theorem of this section is the analogue for $SL_2$ of \cite[Theorem 1.2]{Friedlander:1997}.

\begin{theorem} \label{theorem:extclassesSL2}
For each $r \geq 1$, there exists a rational cohomology class
\[
e_r \in \Ext_{SL_2}^{2p^{r-1}}(k,\gltr)
\]
that restricts nontrivially to $\Ext_{(SL_2)_1}^{2p^{r-1}}(k,\gltr)$.
\end{theorem}

\begin{proof}
Set $G = SL_2$, and consider the Lyndon--Hochschild--Serre (LHS) spectral sequence
\[
E_2^{m,n} = \Ext_{G/G_1}^m(k,\Ext_{G_1}^n(k,\gltr)) \Rightarrow \Ext_G^{m+n}(k,\gltr).
\]
Since $\gl_2$ is self-dual as an $SL_2$-module, the spectral sequence may be rewritten as
\begin{equation} \label{eq:LHSSL2}
E_2^{m,n} = \Ext_{G/G_1}^m(k,\Ext_{G_1}^n(\gltr,k)) \Rightarrow \Ext_G^{m+n}(\gltr,k).
\end{equation}
By \cite[\S 4.1]{Parker:2007} (cf.\ especially the second-to-last paragraph on page 394), this spectral sequence stabilizes at the $E_2$-page. Specifically, taking $i = a = 0$ and $N = (\gl_2)^{(r-1)}$ in the last spectral sequence on page 394 of \cite{Parker:2007} if $p \geq 3$, or in the second spectral sequence on page 395 of \cite{Parker:2007} if $p = 2$, one obtains for $q \geq 0$ the vector space decomposition
\begin{equation} \label{eq:SL2stabilize} \textstyle
\Ext_G^q(\gltr,k) \cong \bigoplus_{n=0}^q E_2^{q-n,n} \cong \bigoplus_{n=0}^q \Ext_G^{q-n}(\gl_2^{(r-1)},\nabla(n)).
\end{equation}

Since \eqref{eq:LHSSL2} stabilizes at the $E_2$-page, one has for each $n \geq 0$ that $E_2^{0,n}$ consists entirely of permanent cycles. On the other hand, by standard properties of the LHS spectral sequence, the vertical edge map
\[
\Ext_G^n(\gltr,k) \rightarrow E_2^{0,n} = \Hom_{G/G_1}(k,\Ext_{G_1}^n(\gltr,k)),
\]
which identifies with the cohomological restriction map from $G$ to $G_1$, has as its image precisely the space of permanent cycles in $E_2^{0,n}$. Thus, one concludes that restriction from $G$ to $G_1$ induces for each $n \geq 0$ a surjective map
\[
\Ext_G^n(\gltr,k) \twoheadrightarrow E_2^{0,n} = \Hom_{G/G_1}(k,\Ext_{G_1}^n(\gltr,k)) \cong \Hom_G(\gl_2^{(r-1)},\nabla(n)).
\]

Now to prove the theorem, it suffices only to show for each $r \geq 1$ that
\[
E_2^{0,2p^{r-1}} \cong \Hom_G(\gl_2^{(r-1)},\nabla(2p^{r-1})) \cong \Hom_{G/G_{r-1}}(\gl_2^{(r-1)},\nabla(2p^{r-1})^{G_{r-1}})
\]
is nonzero. Since $\nabla(n) = \ind_B^G(n)$, where $B$ is the ``negative'' Borel subgroup of $G$ containing $T$, it follows from \cite[I.6.12(a)]{Jantzen:2003} that $\nabla(2p^{r-1})^{G_{r-1}} \cong \nabla(2)^{(r-1)}$ as $G/G_{r-1}$-modules. Then
\[
\Hom_{G/G_{r-1}}(\gl_2^{(r-1)},\nabla(2)^{(r-1)}) \cong \Hom_G(\gl_2,\nabla(2)).
\]
Since $\gl_2$ admits a Weyl filtration, the dimension of this last $\Hom$-space is equal by \cite[II.4.19]{Jantzen:2003} to the multiplicity with which $\Delta(2)$ occurs in such a filtration. Thus, we conclude by the remarks of Section \ref{section:preliminaries} that $\Hom_G(\gl_2,\nabla(2))$ is one-dimensional, and hence that $E_2^{0,2p^{r-1}}$ is nonzero.
\end{proof}

Next we compute the dimension of the cohomology group $\Ext_{SL_2}^{2p^{r-1}}(k,\gltr)$.

\begin{theorem} \label{theorem:onedimext}
Let $r \geq 1$. Then $\Ext_{SL_2}^{2p^{r-1}}(k,\gltr)$ is one-dimensional. In particular, restriction from $SL_2$ to $(SL_2)_1$ defines an isomorphism
\[
\Ext_{SL_2}^{2p^{r-1}}(k,\gltr) \cong \Ext_{(SL_2)_1}^{2p^{r-1}}(k,\gltr)^{SL_2}.
\]
\end{theorem}

\begin{proof}
Set $G = SL_2$ and $q = 2p^{r-1}$. We saw in the proof of Theorem \ref{theorem:extclassesSL2} that
\begin{equation} \label{eq:sawalready}
\begin{split}
\Ext_G^q(k,\gltr) &\cong \textstyle \bigoplus_{n=0}^q \Ext_G^{q-n}(\gl_2^{(r-1)},\nabla(n)) \\
&\cong \textstyle \bigoplus_{n=0}^q \Ext_G^{q-n}(\Delta(n),\gl_2^{(r-1)}),
\end{split}
\end{equation}
and that the $n=q$ summand in this direct sum identifies with $\Ext_{G_1}^q(k,\gltr)^G$ and is one-dimen\-sional. Then to prove the theorem it suffices to show that the other summands in \eqref{eq:sawalready} are zero.

Twisting a module by an interate of the Frobenius morphism is an exact functor, so applying the Frobenius twist $(-)^{(r)}$ to \eqref{eq:goodfiltration}, one obtains the new short exact sequence
\[
0 \rightarrow k \rightarrow \gltr \rightarrow \nabla(2)^{(r)} \rightarrow 0.
\]
One obtains a similar short exact sequence upon applying $(-)^{(r-1)}$ to \eqref{eq:goodfiltration}. Then considering the associated long exact sequences in cohomology, and applying the fact that all higher extensions of induced modules by Weyl modules split \cite[II.4.13]{Jantzen:2003}, it follows that
\begin{equation} \label{eq:Parker4.3} \textstyle
\Ext_G^q(k,\gltr) \cong \Ext_G^q(k,\nabla(2)^{(r)}) \cong \bigoplus_{n=0}^q \Ext_G^{q-n}(\Delta(n),\nabla(2)^{(r-1)}).
\end{equation}
Now if $r = 1$, the summands with $n \neq q$ are all zero by \cite[II.4.13]{Jantzen:2003}, so assume for the remainder of the proof that $r \geq 2$.

Write $[\cdot]$ for the greatest integer function. Then \cite[Theorem 6.1]{Parker:2007} asserts that if $\Delta(n)$ lies in the same block as $\nabla(2)^{(r-1)}$, then
\begin{equation} \label{eq:Parker6.1}
\Ext_G^{q-n}(\Delta(n),\nabla(2)^{(r-1)}) \cong \textstyle \bigoplus_{i=0}^{q-n} \Ext_G^{q-n-i}(\Delta([n/p]+i),\nabla(2)^{(r-2)}).
\end{equation}
Applying this formula recursively, it follows that $\Ext_G^q(k,\nabla(2)^{(r)})$ can be rewritten as a direct sum of various extension groups $\Ext_G^a(\Delta(b),\nabla(2))$ with $a,b \in \N$. Moreover, such an extension group is nonzero only if $a = 0$ and $b = 2$, and when nonzero it is exactly one-dimensional \cite[II.4.13]{Jantzen:2003}.

Let us call the right-hand side of \eqref{eq:Parker4.3} the first step in the recursion to compute $\Ext_G^q(k,\nabla(2)^{(r)})$. Then for $m \geq 1$, the $(m+1)$-th step in the recursion is obtained by applying \cite[Theorem~6.1]{Parker:2007} to all of the nonzero terms from the $m$-th step in the recursion. Thus, $\dim_k \Ext_G^q(k,\nabla(2)^{(r)})$ is equal to the number of times that $\Hom_G(\Delta(2),\nabla(2))$ occurs in the $r$-th step of the recursion. For example, starting with the expression $\Hom_G(\Delta(2p^{r-1}),\nabla(2)^{(r-1)})$, the subsequent steps of the recursion yield the $\Hom$-groups $\Hom_G(\Delta(2p^{r-s}),\nabla(2)^{(r-s)})$ for $2 \leq s \leq r$, resulting finally in a single copy of $\Hom_G(\Delta(2),\nabla(2))$. We claim that this is the only copy of $\Hom_G(\Delta(2),\nabla(2))$ that occurs in the $r$-th step of the recursion.

First, suppose that in some step of the recursion there occurs a nonzero extension group of the form $\Ext_G^a(\Delta(b),\nabla(2)^{(s)})$ with $a+b < 2p^s$ (so by \eqref{eq:Parker4.3}, this expression occurs in the second step of the recursion or beyond). Then from the recursion formula, and recalling the discussion for when a weight can be in the same block as $2p^{s+1}$, it follows that there exists $i \in \N$ such that in the previous step of the recursion, there is a nonzero extension group of the form
\begin{align*}
&\Ext_G^{a+i}(\Delta(p(b-i)),\nabla(2)^{(s+1)}) & \text{if $b-i$ is even, or} \\
&\Ext_G^{a+i}(\Delta(p(b-i)+(p-2)),\nabla(2)^{(s+1)}) & \text{if $b-i$ is odd.}
\end{align*}
Now
\begin{align*}
(a+i)+p(b-i)+p-2 &= (a+b)+(p-1)(b+1)-(p-1)i-1 \\
&< 2p^s +(p-1)(b+1) \\
&\leq 2p^s+(p-1)(2p^s) =2p^{s+1}.
\end{align*}
This shows that, whether $b-i$ is even or odd, the immediate precursor in the recursion to the extension group $\Ext_G^a(\Delta(b),\nabla(2)^{(s)})$ is an extension group of the form $\Ext_G^c(\Delta(d),\nabla(2)^{(s+1)})$ with the property $c+d < 2p^{s+1}$. This is a contradiction, because, working backwards to the first step in the recursion, each extension group $\Ext_G^{q-n}(\Delta(n),\nabla(2)^{(r-1)})$ in the right-hand side of \eqref{eq:Parker4.3} satisfies the property $(q-n)+n = 2p^{r-1}$. Thus, we conclude that at no point in the recursion can there be an extension group of the form $\Ext_G^a(\Delta(b),\nabla(2)^{(s)})$ with $a+b < 2p^s$.

Now consider the possible precursors in the recursion to $\Hom_G(\Delta(2),\nabla(2))$. From the recursion formula, and from the conditions for a weight to be in the same block as $2p$, the only possible immediate precursors to $\Hom_G(\Delta(2),\nabla(2))$ are $\Hom_G(\Delta(2p),\nabla(2)^{(1)})$, $\Ext_G^1(\Delta(p+(p-2)),\nabla(2)^{(1)})$, and $\Ext_G^2(\Delta(0),\nabla(2)^{(1)})$, though the discussion of the previous paragraph shows that the latter two possibilities cannot occur. Similarly, given $s \geq 1$, the possible immediate precursors to $\Hom_G(\Delta(2p^s),\nabla(2)^{(s)})$ are of the form
\begin{align*}
&\Ext_G^i(\Delta(p(2p^s-i)),\nabla(2)^{(s+1)}) & \text{if $2p^s-i$ is even, or} \\
&\Ext_G^i(\Delta(p(2p^s-i)+(p-2)),\nabla(2)^{(s+1)}) & \text{if $2p^s-i$ is odd,}
\end{align*}
for some $i \in \N$. If $i \geq 1$, then
\begin{align*}
i+p(2p^s-i)+(p-2) &= 2p^{s+1} - (p-1)i+(p-2) \\
&\leq 2p^{s+1}-(p-1)+(p-2) \\
&= 2p^{s+1}-1 < 2p^{s+1}.
\end{align*}
Thus, the previous paragraph shows that the only possible immediate precursor in the recursion to $\Hom_G(\Delta(2p^s),\nabla(2)^{(s)})$ is $\Hom_G(\Delta(2p^{s+1}),\nabla(2)^{(s+1)})$, and from this it follows that the only nonzero expressions occurring in the recursion are those already given above, namely, $\Hom_G(\Delta(2p^{r-s}),\nabla(2)^{(r-s)})$ for $0 \leq s \leq r$. This proves the claim, and hence establishes that the only nonzero summand in \eqref{eq:sawalready} is the summand with $n = q$.
\end{proof}

As a corollary of the proof, we have:

\begin{corollary} \label{corollary:extvanishing}
Let $r \geq 1$, and let $0 \leq n < 2p^r$. Then
\[
\Ext_{SL_2}^{2p^r-n}(\Delta(n),\gltr) = \Ext_{SL_2}^{2p^r - n}(\Delta(n),\nabla(2)^{(r)}) = 0.
\]
\end{corollary}

\section{Generalization to \texorpdfstring{$GL_2$}{GL2}}

In this section we show that analogues of Theorems \ref{theorem:extclassesSL2} and \ref{theorem:onedimext} hold with $SL_2$ replaced by $GL_2$. This gives a new existence proof for the universal extension classes for $GL_2$ first constructed by Friedlander and Suslin via the theory of strict polynomial functors. First we require a general result about restriction from $GL_n$ to $SL_n$:

\begin{proposition} \label{proposition:resGLtoSL}
Let $n \geq 2$, and let $V$ and $W$ be rational $GL_n$-modules, considered also as rational $SL_n$-modules by restriction. Suppose that the center of $GL_n$ acts trivially on both $V$ and $W$. Then restriction from $GL_n$ to $SL_n$ defines an isomorphism of rational cohomology groups
\[
\Ext_{GL_n}^\bullet(V,W) \cong \Ext_{SL_n}^\bullet(V,W).
\]
\end{proposition}

\begin{proof}
Given $V$ and $W$, consider the associated LHS spectral sequence
\[
E_2^{i,j} = \Ext_{GL_n/SL_n}^i(k,\Ext_{SL_n}^j(V,W)) \Rightarrow \Ext_{GL_n}^{i+j}(V,W).
\]
The quotient group $GL_n/SL_n$ is isomorphic to the multiplicative group $\G_m$, a diagonalizable group scheme. Then by \cite[I.4.3]{Jantzen:2003}, one has $E_2^{i,j} = 0$ for all $i > 0$, so the spectral sequence collapses to the first column, and it follows that restriction from $GL_n$ to $SL_n$ induces an isomorphism
\[
\Ext_{GL_n}^\bullet(V,W) \cong \Ext_{SL_n}^\bullet(V,W)^{GL_n/SL_n}.
\]
Then to finish the proof, it suffices to show that the action of $GL_n$ on $\Ext_{SL_n}^\bullet(V,W)$ is trivial. Since taking fixed points commutes with scalar extension \cite[I.2.10(3)]{Jantzen:2003}, we may assume for the rest of the proof that the field $k$ is algebraically closed. In this case $GL_n(k)$ is dense in $GL_n$ \cite[I.6.16]{Jantzen:2003}, so it suffices further to show that the action of $GL_n(k)$ on $\Ext_{SL_n}^\bullet(V,W)$ is trivial.

Write $Z$ for the center of $GL_n$. Recall that the action of $GL_n$ on $\Ext_{SL_n}^\bullet(V,W)$ is induced by the conjugation action of $GL_n$ on $SL_n$ together with the given actions of $GL_n$ on $V$ and $W$ \cite[I.6.7]{Jantzen:2003}. Also, $SL_n$ acts trivially on $\Ext_{SL_n}^\bullet(V,W)$ for any pair of rational modules $V$ and $W$. Since the conjugation action of $Z$ on $SL_n$ is trivial, and since $Z$ acts trivially on $V$ and $W$ by assumption, the induced action of $Z$ on $\Ext_{SL_n}^\bullet(V,W)$ is trivial. But, using the fact that $k$ contains arbitrary $n$-th roots, $GL_n(k)$ is generated as an abstract group by $Z(k)$ and $SL_n(k)$. Specifically, if $A \in GL_n(k)$ and if $a \in k$ is any fixed $n$-th root of $\det(A)$, then $A = (A \cdot a^{-1} I_n) \cdot (aI_n) \in SL_n(k) \cdot Z(k)$, where $I_n$ denotes the $n \times n$ identity matrix. Thus, we conclude that $GL_n(k)$ acts trivially on $\Ext_{SL_n}^\bullet(V,W)$, and hence so does the full group scheme $GL_n$.
\end{proof}

\begin{remark} \label{remark:resGLn1toSLn1}
Since $(GL_n)_1/(SL_n)_1 \cong (\G_m)_1$ is a diagonalizable group scheme, an argument completely analogous to that in the above proof shows that if $V$ and $W$ are rational $(GL_n)_1$-modules, then restriction from $(GL_n)_1$ to $(SL_n)_1$ defines an isomorphism
\[
\Ext_{(GL_n)_1}^\bullet(V,W) \cong \Ext_{(SL_n)_1}^\bullet(V,W)^{(GL_n)_1/(SL_n)_1}.
\]
In particular, restriction from $(GL_n)_1$ to $(SL_n)_1$ is injective.
\end{remark}

\begin{remark}
It is not true in general that if $V$ and $W$ are rational $GL_n$-modules that restriction from $GL_n$ to $SL_n$ defines an isomorphism of rational cohomology groups. Indeed, choose rational $GL_n$-modules $V$ and $W$ such that $\Ext_{SL_n}^\bullet(V,W)$ is nonzero, and write $\det$ for the $1$-dimensional determinant representation of $GL_n$. Then $W \otimes \det \cong W$ as rational $SL_n$-modules, and
\begin{equation} \label{eq:tensordet}
\Ext_{SL_n}^\bullet(V,W \otimes \det) \cong \Ext_{SL_n}^\bullet(V,W) \otimes \det
\end{equation}
as rational $GL_n$-modules. The action of $GL_n$ on the cohomology spaces in \eqref{eq:tensordet} factors through the quotient $GL_n/SL_n \cong \G_m$. Now $GL_n$ will act nontrivially on at least one of the cohomology spaces $\Ext_{SL_n}^\bullet(V,W)$ and $\Ext_{SL_n}^\bullet(V,W) \otimes \det$, so either $\Ext_{GL_n}^\bullet(V,W) \not\cong \Ext_{SL_n}^\bullet(V,W)$ or $\Ext_{GL_n}^\bullet(V,W \otimes \det) \not\cong \Ext_{SL_n}^\bullet(V,W \otimes \det)$, because the action of $GL_n$ on $\Ext_{GL_n}^\bullet(V,W)$ and $\Ext_{GL_n}^\bullet(V,W \otimes \det)$ is trivial.
\end{remark}

The center of $GL_n$ acts trivially on the adjoint representation, hence also on any Frobenius twist of it. Then a special case of Proposition \ref{proposition:resGLtoSL} is:

\begin{corollary} \label{corollary:gltrresiso}
Let $r \geq 1$. Then restriction from $GL_n$ to $SL_n$ defines an isomorphism
\[
\Ext_{GL_n}^\bullet(k,\glnr) \cong \Ext_{SL_n}^\bullet(k,\glnr).
\]
\end{corollary}

We now obtain the main theorem of this section:

\begin{theorem} \label{theorem:extclassesGL2}
Let $r \geq 1$. Then $\Ext_{GL_2}^{2p^{r-1}}(k,\gltr)$ is one-dimensional, and restriction from $GL_2$ to $(GL_2)_1$ defines an isomorphism
\[
\Ext_{GL_2}^{2p^{r-1}}(k,\gltr) \cong \Ext_{(GL_2)_1}^{2p^{r-1}}(k,\gltr)^{GL_2}.
\]
In particular, any nonzero class $e_r \in \Ext_{GL_2}^{2p^{r-1}}(k,\gltr)$ restricts nontrivially to $\Ext_{(GL_2)_1}^{2p^{r-1}}(k,\gltr)$.
\end{theorem}

\begin{proof}
Consider the following commutative diagram of restriction homomorphisms:
\[
\xymatrix{
\Ext_{GL_2}^{2p^{r-1}}(k,\gltr) \ar@{->}[d]^{\delta} \ar@{->}[r]^{\alpha} & \Ext_{SL_2}^{2p^{r-1}}(k,\gltr) \ar@{->}[d]^{\beta} \\
\Ext_{(GL_2)_1}^{2p^{r-1}}(k,\gltr)^{GL_2} \ar@{->}[r]^{\gamma} & \Ext_{(SL_2)_1}^{2p^{r-1}}(k,\gltr)^{SL_2}.
}
\]
We know that $\alpha$ and $\beta$ are isomorphisms by Corollary \ref{corollary:gltrresiso} and Theorem \ref{theorem:onedimext}, and that $\gamma$ is an injection by Remark \ref{remark:resGLn1toSLn1}. Since $\beta \circ \alpha$ is a $GL_2$-equivariant isomorphism, we conclude that
\begin{align*}
\Ext_{(SL_2)_1}^{2p^{r-1}}(k,\gltr)^{SL_2} &= \Ext_{(SL_2)_1}^{2p^{r-1}}(k,\gltr)^{GL_2} \\
&= ( \Ext_{(SL_2)_1}^{2p^{r-1}}(k,\gltr)^{(GL_2)_1} )^{GL_2/(GL_2)_1} \\
&\cong \Ext_{(GL_2)_1}^{2p^{r-1}}(k,\gltr)^{GL_2},
\end{align*}
where the last isomorphism follows from Remark \ref{remark:resGLn1toSLn1}. Thus, we conclude that $\gamma$ is an isomorphism, from which it follows that $\delta$ is an isomorphism as well. This proves the second and third claims of the theorem, while the first follows from the isomorphism $\alpha$ and Theorem \ref{theorem:onedimext}.
\end{proof}

\section{Cohomological finite-generation}

In this section we describe how Theorem \ref{theorem:extclassesGL2} can be applied to show that if $G$ is an infinitesimal subgroup scheme of $GL_2$, then the cohomology ring $\Hbul(G,k)$ is a finitely-generated $k$-algebra. The arguments establishing these facts are essentially the same as those given by Friedlander and Suslin in \cite[\S 1]{Friedlander:1997}, so we will omit most of the details. The one case where slightly modified arguments are required is the case $p = 2$. This is because when working with infinitesimal subgroups of $GL_n$, Friedlander and Suslin specifically assume $n \not\equiv 0 \mod p$. To deal with the case $n \equiv 0 \mod p$, they instead embed their given infinitesimal group $G$ into a larger general linear group where their general arguments go through. As we have not exhibited universal extension classes for larger general linear groups, this workaround is not available to us and we must argue more directly.

Before describing the necessary modifications to handle the case $p = n = 2$, we give some notation in order to state the main theorem. For each $i \geq 1$, let $e_i \in \opH^{2p^{i-1}}(GL_2,\gl_2^{(i)})$ be a fixed nonzero cohomology class as in Theorem \ref{theorem:extclassesGL2}. Then as in \cite[Remark 1.2.2]{Friedlander:1997}, we can for each $j \geq 1$ pull back $e_i$ along the $j$-th iterate of the Frobenius morphism to obtain a new cohomology class $e_i^{(j)} \in \opH^{2p^{i-1}}(GL_2,\gl_2^{(i+j)})$. Now let $r \geq 1$, and let $G \subset (GL_2)_r$ be an infinitesimal subgroup scheme. Then the restriction of $\gltr$ to $G$ is trivial, so
\[
\Hbul(G,\gltr) \cong \Hbul(G,k) \otimes \gltr \cong \Hom_k((\gltr)^*,\Hbul(G,k)).
\]
Thus, the restriction of $e_i^{(r-i)}$ to $G$ determines a linear map
\begin{equation} \label{eq:linearmap}
(\gltr)^* \rightarrow \opH^{2p^{i-1}}(G,k),
\end{equation}
and hence a homomorphism of graded $k$-algebras
\begin{equation} \label{eq:morphismgradedalgebras}
S^\bullet((\gltr)^*(2p^{i-1})) \rightarrow \Hbul(G,k).
\end{equation}
Here $S^\bullet((\gltr)^*(2p^{i-1}))$ denotes the symmetric algebra on $(\gltr)^*$, considered as a graded algebra with $(\gltr)^*$ concentrated in degree $2p^{i-1}$. Note that if $G = (GL_2)_r$, then \eqref{eq:linearmap} and \eqref{eq:morphismgradedalgebras} are $GL_2$-equivariant homomorphisms. Taking the product of the maps in \eqref{eq:morphismgradedalgebras}, one obtains a homomorphism of graded algebras $\bigotimes_{i=1}^r S^\bullet((\gltr)^*(2p^{i-1})) \rightarrow \Hbul(G,k)$. The main theorem of this section is now:

\begin{theorem} \label{theorem:GL2fg}
Let $G \subset (GL_2)_r$ be an infinitesimal group scheme over $k$. Let $C$ be a commutative $k$-algebra, considered as a trivial $G$-module, and let $M$ be a noetherian $C$-module on which $G$ acts by $C$-linear transformations. Then $\Hbul(G,M)$ is noetherian over $\bigotimes_{i=1}^r S^\bullet((\gltr)^*(2p^{i-1})) \otimes C$. In particular, $\Hbul(G,k)$ is a finite module over the polynomial algebra $\bigotimes_{i=1}^r S^\bullet((\gltr)^*(2p^{i-1}))$, and hence is a finitely-generated noetherian $k$-algebra.
\end{theorem}

The proof of the theorem is by induction on $r$. For $r=1$, the key step is to show that the linear map \eqref{eq:linearmap} coincides, up to a nonzero scalar factor, with the composition of maps
\begin{equation} \label{eq:compositionofmaps}
(\gl_2^{(1)})^* \twoheadrightarrow (\g^{(1)})^* \rightarrow \opH^2(G,k),
\end{equation}
Here $\g = \Lie(G)$, the arrow $(\gl_2^{(1)})^* \twoheadrightarrow (\g^{(1)})^*$ is induced by duality from the inclusion of Lie algebras $\g \hookrightarrow \gl_2$, and the map $(\g^{(1)})^* \rightarrow \opH^2(G,k)$ is the horizontal edge map of the May spectral sequence 
\begin{equation} \label{eq:Mayspecseq}
E_2^{s,t} = S^{s/2}((\g^{(1)})^*) \otimes \opH^t(\g,k) \Rightarrow \opH^{s+t}(G,k);
\end{equation}
cf.\ \cite[p.\ 217--218]{Friedlander:1997}. By the naturality of the May spectral sequence, it suffices to check that the maps coincide in the special case $G = (GL_2)_1$. In this case, \eqref{eq:linearmap} and \eqref{eq:compositionofmaps} are both $GL_n$-equivariant homomorphisms. To check that \eqref{eq:linearmap} and \eqref{eq:compositionofmaps} coincide in the special case $G = (GL_2)_1$, we have the following analogue of \cite[Lemma 1.7]{Friedlander:1997} that is valid for fields of arbitrary positive characteristic:

\begin{lemma} \label{lemma:morphismcoincides}
Let $r \geq 1$, and let $(\gltr)^* \rightarrow \opH^{2p^{r-1}}((GL_2)_1,k)$ be a nonzero $GL_2$-equivariant homomorphism. Then this homomorphism coincides, up to a nonzero scalar factor, with the composition
\begin{equation} \label{eq:composition}
(\gltr)^* \rightarrow S^{p^{r-1}}((\gl_2^{(1)})^*) \rightarrow \opH^{2p^{r-1}}((GL_2)_1,k),
\end{equation}
where the first arrow raises elements to the $p^{r-1}$ power, and the second arrow is the horizontal edge map of the May spectral sequence.
\end{lemma}

\begin{proof}
The map \eqref{eq:composition} is evidently a $GL_2$-module homomorphism. On the other hand,
\[
\Hom_{GL_2}((\gltr)^*,\opH^{2p^{r-1}}((GL_2)_1,k)) \cong \opH^{2p^{r-1}}((GL_2)_1,\gltr)^{GL_2}
\]
is one-dimensional by Theorem \ref{theorem:extclassesGL2}, so any two nonzero $GL_2$-module homomorphisms $(\gltr)^* \rightarrow \opH^{2p^{r-1}}((GL_2)_1,k)$ must coincide up to a scalar factor.
\end{proof}

With Lemma \ref{lemma:morphismcoincides} established, the case $r=1$ of Theorem \ref{theorem:GL2fg} is proven exactly as in \cite[p.~217--218]{Friedlander:1997}, and then the general case of Theorem \ref{theorem:GL2fg} is handled exactly as in \cite[p.~219--220]{Friedlander:1997}. Repeating the arguments of \cite[p.~220--221]{Friedlander:1997}, the reader can also obtain a new proof that if $G$ is a finite group scheme such that for some finite field extension $k'$ of $k$, the connected component of $G_{k'}$ containing the identity element is an infinitesimal subgroup scheme of $GL_2$, then the cohomology ring $\Hbul(G,k)$ is a finitely-generated $k$-algebra, and $\Hbul(G,M)$ is a finitely-generated $\Hbul(G,k)$-module for any rational $G$-module $M$.

\section*{Acknowledgements}

This work was completed as a result of conversations with Alison Parker at the American Institute of Mathematics workshop on Cohomology Bounds and Growth Rates in June 2012. The author thanks the organizers of the workshop for inviting him to participate.


\begin{thebibliography}{1}

\bibitem{Doty:2005}
S.~Doty and A.~Henke, \emph{Decomposition of tensor products of modular
  irreducibles for {${\rm SL}\sb 2$}}, Q. J. Math. \textbf{56} (2005), no.~2,
  189--207.

\bibitem{Friedlander:1997}
E.~M. Friedlander and A.~Suslin, \href{http://dx.doi.org/10.1007/s002220050119}
  {\emph{Cohomology of finite group schemes over a field}}, Invent. Math.
  \textbf{127} (1997), no.~2, 209--270.

\bibitem{Jantzen:2003}
J.~C. Jantzen, \emph{Representations of algebraic groups}, second ed.,
  Mathematical Surveys and Monographs, vol. 107, American Mathematical Society,
  Providence, RI, 2003.

\bibitem{Parker:2007}
A.~E. Parker, \href{http://dx.doi.org/10.1016/j.aim.2006.05.015} {\emph{Higher
  extensions between modules for {$\rm SL_2$}}}, Adv. Math. \textbf{209}
  (2007), no.~1, 381--405.

\bibitem{Kallen:2004}
W.~van~der Kallen, \emph{Cohomology with {G}rosshans graded coefficients},
  Invariant theory in all characteristics, CRM Proc. Lecture Notes, vol.~35,
  Amer. Math. Soc., Providence, RI, 2004, pp.~127--138.

\end{thebibliography}

\providecommand{\bysame}{\leavevmode\hbox to3em{\hrulefill}\thinspace}
\providecommand{\MR}{\relax\ifhmode\unskip\space\fi MR }
\providecommand{\MRhref}[2]{%
  \href{http://www.ams.org/mathscinet-getitem?mr=#1}{#2}
}
\providecommand{\href}[2]{#2}

\end{document}